\documentclass[11pt,hyp,]{nyjm}
\usepackage{hyperref}
\hypersetup{nesting=true,debug=true,naturalnames=true}
\usepackage{graphicx,amssymb,upref}
\usepackage{amsmath, amsthm, amscd, amsfonts}

\usepackage{tikz}
\usetikzlibrary{matrix,arrows}

\let\<\langle
\let\>\rangle

\let\uml\"

\newtheorem{thm}{Theorem}[section]
\newtheorem{lem}[thm]{Lemma}
\newtheorem{prop}[thm]{Proposition}
\newtheorem{cor}[thm]{Corollary}

\theoremstyle{definition}

\theoremstyle{remark}
\newtheorem{rem}[thm]{Remark}
\numberwithin{equation}{section}

\newcommand{\h}{\mathcal{H}}

\newcommand{\C}{{\mathcal{C}}}
\newcommand{\A}{\mathcal{A}}

\newcommand{\B}{{\mathcal{B}}}

\title{Smooth cohomology of $ C^* $-algebras}

\author[M. Amini]{Massoud Amini}
\address{Department of Mathematics, Tarbiat Modares University, P.O Box 14115-175, Tehran, Iran}
\email{mamini@modares.ac.ir}
\author[A. Shirinkalam]{Ahmad Shirinkalam}
\address{Department of Basic sciences, IAU Central Tehran Branch,    Tehran, Iran} 
\email{shirinkalam\_a@aut.ac.ir}



\begin{document}

\begin{abstract}
We define a notion of smooth cohomology for $ C^* $-algebras which admit a faithful trace. We show that if $ \A\subseteq B(\h) $ is a $ C^* $-algebra with a faithful normal trace $ \tau $ on   the ultra-weak closure $ \bar{\A} $ of $ \mathcal{A} $, and  $ X $ is a normal dual operatorial $ \bar{\A}$-bimodule, then  the first smooth cohomology  $ 	\mathcal{H}^1_{s}(\mathcal{A},X) $ of $ \mathcal{A} $ is equal to $ \mathcal{H}^1(\mathcal{A},X_{\tau})$, where $ X_{\tau} $ is a closed submodule of $ X $  consisting of smooth elements.
\end{abstract}
\maketitle
\tableofcontents



%
%
%
%

%

\section{Introduction}

Hochschild cohomology  is an important invariant for Banach and operator algebras.
George Elliott used this along with K-theory groups in the  classification of separable AF $C^*$-algebras \cite{el}.
Also, Alain Connes and Uffe Haagerup characterized the injectivity and hyperfiniteness of a von Neumann algebras by the vanishing of its cohomology group over all dual normal modules \cite{c1},\cite{c2},\cite{c3},\cite{h}. Another example is the proof of equivalence of amenability and nuclearity for $ C^* $-algebras,  by Alain Connes (1978) (amenable $ \Rightarrow $ nuclear)  and Uffe Haagerup (1983) (nuclear $ \Rightarrow $ amenable).

A study of cohomology in the algebraic setting was initiated by Hochschild (1945-47) \cite{h1},\cite{h2},\cite{h3}.
After Kaplansky (1953), we know that various questions about the properties of  derivations on $ C^* $-algebras and von Neumann algebras could be translated into  certain cohomology groups being equal to each other (or to zero).
Following R. V. Kadison \cite{k}, S. Sakai (1966) showed that every derivation $ \delta:\mathcal{M}\rightarrow \mathcal{M}  $ on a von Neumann algebra $ \mathcal{M} $ is inner, which is equivalent to vanishing of
 the first continuous cohomology group
 $ \mathcal{H}^1(\mathcal{M},\mathcal{M}) $ \cite{s}. When $ \mathcal{M} $ is faithfully represented on a Hilbert space $ \h $, $ B(\h) $, the space of all bounded linear operators on $\h $, becomes an $ \mathcal{M}$-bimodule and the cohomology groups $ \mathcal{H}^n(\mathcal{M},B(\h)) $ are defined. In all known cases these groups are zero, but in general we do not know what happens.

B. E. Johnson, R.V. Kadison and J. R. Ringrose (1972) showed that if $ \mathcal{M} $ is hyperfinite and
$ X $ is an arbitrary dual normal  $ \mathcal{M} $-bimodule, then $ \mathcal{H}^n(\mathcal{M},X)=0 $ for all $ n>0$ \cite{jkr}.
Later,  E. Christensen, E. G. Effros and A. M. Sinclair (1987) considered complete boundedness of maps and
 applied operator space techniques to cohomology of operator algebras. This method worked perfectly for von Neumann algebras of types $ \mathrm{I},\mathrm{II}_\infty $
and $ \mathrm{III }$ \cite{ces}. Type $ \mathrm{I} $ can be handled by hyperfiniteness results while types $ \mathrm{II}_\infty $ and $ \mathrm{III} $ are stable under tensoring with $ B(\h) $ which is enough to obtain complete boundedness of cocycles, but not all type
$ \mathrm{II}_1 $ factors have this property. Some partial results for $ \mathrm{II}_1 $ algebras obtained by F. Pop and R. R. Smith (1994) \cite{ps}. For example if $ M $ is a separably acting type $ \mathrm{II}_1 $ von Neumann algebra with a Cartan subalgebra, then
$ \mathcal{H}^n(\mathcal{M},\mathcal{M})=0 $ for all $ n>0. $

The  case of $ \mathrm{II}_1 $ factor studied by S. Popa and S. Vaes (2014) (for the continuous
 $ L^2 $-cohomology) \cite{pv} and A. Galatan and Popa (2017) (for factors with some additional conditions) \cite{gp}.

In the latter paper, the authors related the so-called smooth cohomology of a von Neumann algebra with coefficients in a Banach module $X$ with the ordinary cohomology with coefficients in the  smooth part of $X$ (which is a closed submodule of $X$) and  showed that for factors, each derivation  with values in the smooth part is inner.

The main objective of this paper is to handle the same correspondence for $C^* $-algebras.
 Following \cite{gp}, we define a notion of smooth cohomology for a $C^* $-algebra $ \A $ with a faithful  trace.
The main result of the paper asserts that the smooth cohomology of $ \A $ with coefficients in $X$  and the Hochschild cohomology of $ \A $  with coefficients in the smooth part of $X$ are the same.
In order to do this, we show that the smooth weak continuous cocycles on   $ \A $
can be extended to its ultra-weak closure  $ \bar{\A} $, without changing the cohomology groups. The precise statement is as follows:

\begin{thm}\label{t1}
	Let $ \A\subseteq B(\h) $ be a $ C^* $-algebra with a faithful normal trace on the ultra-weak closure   $ \bar{\A} $ of $\A  $ and let $ X $ be a normal dual $ \bar{\A} $-bimodule.
	Then, for every $ n \in  \mathbb{N}$ we have $$ \mathcal{H}^n_{sw}(\mathcal{A},X)=\mathcal{H}^n_{sw}(\bar{\A},X). $$
\end{thm}

The key point here is that every smooth map on $ \A $ can be extended to $ \bar{\A} $. This will be checked in Lemma \ref{L3}.
 Then, using Proposition \ref{p}  and Lemma \ref{5},  we show that the smooth normal cohomology of $ \A $ coincides with the smooth cohomology of $ \A $:

\begin{thm}\label{t2}
	Let $ \A\subseteq B(\h) $ be a $ C^* $-algebra with a faithful normal trace on   $ \bar{\A} $ and let $ X $ be a normal dual $ \bar{\A} $-bimodule. Then, for every $ n \in  \mathbb{N}$ we have
	$$ \mathcal{H}^n_{sw}(\mathcal{A},X)=\mathcal{H}^n_{s}(\A,X). $$
\end{thm}
This is done by the averaging techniques described in \cite[Section 3]{SS}.
The averaging technique used here is to integrate over the compact unitary group of a finite dimensional $ C^* $-algebra. Taking suitable weak limits as the finite dimensional algebras increase in size leads to averages
that are essentially over infinite dimensional algebras. This method described in an abstract setting by Johnson, Kadison and Ringrose in \cite{jkr}.

	Combining Theorem \ref{t1} with Theorem \ref{t2}, we deduce the following equality (Corollary \ref{c1}):
	$$ \mathcal{H}^n_{s}(\mathcal{A},X)=\mathcal{H}^n_{sw}(\bar{\A},X). $$

In the case when $ X $ is a normal dual operatorial $ \bar{\A} $-bimodule (in the sense of \cite{gp}), we get the main result of the paper:
\begin{thm} \label{main}
	Let $ \A\subseteq B(\h) $ be a $ C^* $-algebra with a faithful normal trace $ \tau $ on   $ \bar{\A} $, the ultra-weak closure of $ \mathcal{A} $, and let $ X $ be a normal dual operatorial $ \bar{\A} $-bimodule. Then,
	$ 	\mathcal{H}^1_{s}(\mathcal{A},X)=\mathcal{H}^1(\mathcal{A} ,X_{\tau}). $
	
\end{thm}

An example of a normal dual operatorial $ \bar{\A} $-bimodule is  $ B(\h) $, the space of all bounded linear operators on a Hilbert space $\h  $ on which $ \A $ is represented. The smooth part of this module is a hereditary $ C^* $-subalgebra of $ B(\h) $ that contains the space of compact operators $ K(\h) $ and a large variety of non-compact smooth elements in general \cite{gp}.

\section{Preliminaries}

Throughout the paper, $ \A_1 $ denotes the closed unit ball of a $C^*$-algebra $ \A $. Also,
the weak (respectively, strong and ultra-weak) operator topology on $ B(\h) $ is denoted by WOT (respectively, SOT  and UWOT).

Let $\A$ be a unital $C^*$-algebra. A positive linear functional $ \tau $ on $\A$ is called
tracial (or a finite trace) if $ \tau(ab)=\tau(ba) $ for all $a,b \in \A$.
A trace on $\A$ is called faithful if $a=0$, whenever $\tau(a^*a)=0$ for every $a\in \A$.
Each faithful trace on $\A$ induces a norm $ \Arrowvert . \Arrowvert_{\tau} $ on $\A$  defined by $\Arrowvert a \Arrowvert^2_{\tau}=\tau(a^*a),\ (a\in \A)$.

Let $\A$ be a  $C^*$-algebra with a  faithful trace $ \tau $ and let $ B $ be a Banach space.
An $ n $-linear map $ T:\A \rightarrow B $ is called \textit{smooth} if it is  continuous relative to the $ \Arrowvert . \Arrowvert_{\tau} $-topology on $\A_1$ and the norm topology on $ B $. A multi-linear map is smooth if it is smooth, separately in each argument.

Let $X$ be a Banach $\A$-bimodule. An element $x\in X$ is called \textit{smooth} if the module maps $\A \rightarrow X;a \mapsto a\cdot x$ and $a\mapsto x\cdot a$ are smooth.
We denoted by $X_\tau$ the closed submodule of all smooth elements in $X$. If $\B$ is a $C^*$-subalgebra of $\A$, then we have $X^A_\tau \subseteq  X^B_\tau$ \cite{gp}.

The Banach $\A$-bimodule  $ X $ is said to be {\it dual} if it is the dual of a Banach space and for each $ a\in \A $, the  maps $X \rightarrow X; x\mapsto a\cdot x  $ and $ x\mapsto x\cdot a  $
are weak* continuous. If in addition, $ \A $ admits a weak* topology (for example whenever $ \A $ is a von Neumann algebra), and for every $ x\in X $ the  maps $\A  \rightarrow X; a\mapsto a\cdot x  $ and $ a\mapsto x\cdot a  $ are weak* continuous, then $ X $ is said to be {\it normal}.

We put $ BL^0(\A,X)=X $, and for each $ n \in \mathbb{N} $, we denote by $ BL^n(\A,X) $ the space of all bounded $ n $-linear maps from $ \A^n $ into $ X. $ The subscript ``$ s $" (respectively, ``$ sw $") means that the maps are smooth (respectively, smooth and separately UWOT-continuous). Let $ \B $ be a subalgebra of $ \A $. An element $ T $ of $  BL^n(\A,X)  $  is called $ \B $-{\it modular}  if for each $ a_1,..., a_{n} \in  \A $ and $ b \in \B $ we have
\begin{center}
	$b\cdot T(a_1,..., a_{n})=T(ba_1,..., a_{n}),$
	$ T(a_1,...,a_jb,a_{j+1},..., a_{n})=T(a_1,...,a_j,ba_{j+1},..., a_{n}), $
	$ T(a_1,..., a_{n}b)=T(ba_1,..., a_{n})b. $
\end{center}
The space of all the $ \B $-modular maps  is denoted by $ BL^n(\A,X:\B) $.

For each $ n>0 $,  the coboundary operators $$ \delta^{n}: BL^n(\A,X) \rightarrow BL^{n+1}(\A,X)$$ are defined by
\begin{align*}
(\delta^nT)(a_1,..., a_{n+1})
&:=a_1\cdot T(a_2,..., a_{n+1})\\
&+\sum_{k=1}^{n}(-1)^kT(a_1,..., a_ka_{k+1},..., a_{n+1})\\
&+(-1)^{n+1}T(a_1,..., a_{n})\cdot a_{n+1}, \quad (a_1,..., a_{n+1}\in \mathcal{A}),
\end{align*}
 and $\delta^0:X\to BL(\mathcal{A},X)$ is defined by $\delta^0(x)(a)=
a\cdot x-x\cdot a$. We have
the cochain complex
$$\{0\}\rightarrow X\xrightarrow{\delta^{0}} BL(\A,X)\cdot \cdot \cdot \xrightarrow{\delta^{n-1}}BL^n(\A,X) \xrightarrow{\delta^{n}}BL^{n+1}(\A,X)\xrightarrow{\delta^{n+1}}\cdot \cdot \cdot,  $$
 called the Hochschild cochain complex.

 Letting $ \mathcal{Z}^n(\mathcal{A},X) =\ker \delta^{n} $ and $\mathcal{B}^n(\mathcal{A},X)= \text{ran} \:\delta^{n} $,
 we  have the quotient linear space
  $$\mathcal{H}^n(\mathcal{A},X):=\mathcal{Z}^n(\mathcal{A},X)/\mathcal{B}^n(\mathcal{A},X),\; \mathcal{H}^0(\mathcal{A},X)=\{x\in X:a\cdot x=x \cdot a (a\in\mathcal{A})\},$$ called the $ n$-th Hochschild cohomology of 	$\mathcal{A}$
with coefficients in $X$.

Following \cite{gp} and \cite{SS},
we may use the  subscripts ``$ s $" and ``$ sw $" in $ BL^{n}_{s}(\A,X) $ and $ BL^{n}_{sw}(\A,X) $). For example,
  $ \mathcal{Z}_{s}^1(\mathcal{A},X) $ is the space of smooth derivations on $ \A $ to $ X $ and $ \mathcal{B}_{s}^1(\mathcal{A},X) $ is the space of inner derivations that is implemented by a smooth element of $ X $.

\section{Smooth cohomology }
In this section we  explore the relation between  $H^1_{s}(\A,X)$ and $H^1(\A,X_{\tau})$.
Let $ \A\subseteq B(\h) $ be a $ C^* $-algebra with a faithful normal trace on $ \A^{\prime\prime} $. By \cite[Theorem 1.2.4]{ren}, on a bounded ball of $ \A$, the WOT, SOT and  UWOT agree. Also, the $ \Arrowvert . \Arrowvert_{\tau} $-topology agrees with  SOT (and also with UWOT) on any bounded subset of $ \A $ by \cite[III. 2.2.17]{bl}. In particular, a bounded net $ (a_i)\subseteq \A $ converges to zero strongly if and only if $ \Arrowvert a_i \Arrowvert_{\tau}\rightarrow 0 $. We use this facts several times. The  results of this section  adapt ideas and techniques from \cite{SS}.

\begin{lem}\label{L2}
	Let $ \A $ and $\B$ be two $ C^* $-subalgebras of $ B(\h) $ and let $ \tau $ be a faithful normal trace on the von Neumann algebra generated by  $ \A $ and $ \B $. Let $\varphi:\A \times\B \rightarrow \mathbb{C} $ be a bounded bilinear smooth form which is separately UWOT-continuous. Then $ \varphi $ extends uniquely to a separately UWOT-continuous, smooth bilinear form $\bar{\varphi}:\bar{\A} \times\B \rightarrow \mathbb{C} $, where $ \bar{\A} $ is the UWOT-closure of $ \A $.
\end{lem}
\begin{proof}
For a fixed $ b\in \B $, the bounded linear functional $ \varphi_b(a):=\varphi (a,b)$ is smooth and UWOT-continuous, so it extends to an UWOT-continuous linear functional $ \widetilde{\varphi_b} :\bar{\A}  \rightarrow \mathbb{C}$. Kaplansky density theorem implies that $ \Arrowvert  \widetilde{\varphi_b} \Arrowvert = \Arrowvert \varphi_b  \Arrowvert. $ Hence, the map
$ \widetilde{\varphi}:\B\rightarrow(\bar{\A})_*;\ b\mapsto \widetilde{\varphi_b} $
is linear and bounded with $ \Arrowvert \widetilde{\varphi}   \Arrowvert \leq \Arrowvert  \varphi  \Arrowvert $. Since $ \varphi $ is UWOT-continuous in the second argument, $ \widetilde{\varphi} $ is  continuous in UWOT on $\B$ and in
$ \sigma((\bar{\A})_*,\A ) $ on $ (\bar{\A})_* $. By \cite[Theorem 5.4]{tak} or \cite[Corollary II.9]{ak}, $ \widetilde{\varphi}(\B_1) $ is relatively $ \sigma((\bar{\A})_*,\bar{\A} ) $-compact in
$ (\bar{\A})_* $, hence  $ \sigma((\bar{\A})_*,\bar{\A} ) $  coincides with the coarser topology $ \sigma((\bar{\A})_*,\A ) $. Combining this with the continuity of $ \widetilde{\varphi} $   yields that $ \widetilde{\varphi} $ is continuous on $ \B_1 $ in UWOT into $ (\bar{\A})_* $ in $ \sigma((\bar{\A})_*,\bar{\A} )  $.
Thus, for each fixed $ a \in \bar{\A}  $, the linear functional $ b \mapsto \widetilde{\varphi_b}(a) $ is UWOT-continuous on $ \B_1 $ and hence, on $ \B. $ Now the bounded bilinear form $\bar{\varphi}:\bar{\A} \times\B \rightarrow \mathbb{C} $ defined by $ \bar{\varphi}(a,b)= \widetilde{\varphi_b}(a) $ is separately UWOT-continuous. It remains to show that $ \bar{\varphi} $ is smooth. The $ \Arrowvert . \Arrowvert_{\tau} $-continuity of $ \bar{\varphi} $ on $ \B_1 $ follows from the continuity of $ \varphi $. For the first argument of $\bar{\varphi}  $, it is enough to show that $ \widetilde{\varphi_b} :\bar{\A}  \rightarrow \mathbb{C}$ is smooth. Since  $ \widetilde{\varphi_b} $ is UWOT-continuous on $ (\bar{\A})_1 $,  it is also $ \Arrowvert . \Arrowvert_{\tau} $-continuous.
\end{proof}
\begin{lem}\label{L3}
	Let $ \A\subseteq B(\h) $ be a $ C^* $-algebra with a faithful normal trace on $ \A^{\prime\prime} $ and let $ X $ be a dual module with predual $ X_*. $ If
	$ \varphi:\A \times \A \times... \times \A \rightarrow X $  is a bounded $ n $-linear smooth map which is separately UWOT-weak*-continuous, then it extends  uniquely (without change of norm) to a separately UWOT-continuous, smooth $ n $-linear map
	$ \bar{\varphi}:\bar{\A} \times \bar{\A}\times ... \times \bar{\A} \rightarrow X $.
\end{lem}
\begin{proof}
	We give the proof in two cases;
	
	Case 1. Let $ X=\mathbb{C} $. We will construct a finite sequence $ \varphi = \varphi_0, \varphi_1, ..., \varphi_n $ of bounded $ n $-linear functionals with the following properties:
	
	(i) $ \varphi_k :\underbrace{\bar{\A} \times \bar{\A}\times ... \times \bar{\A}}_{k-times} \times \A \times ... \times \A \rightarrow \mathbb{C} $,
	
	(ii) $ \varphi_k  $ extends $ \varphi_{k-1} $ without change of norm,
	
	(iii) $ \varphi_k  $ is separately UWOT-continuous,
	
	(iv) $ \varphi_k  $ is a smooth map.
	
This proves the existence of $ \bar{\varphi}=\varphi_n $. The uniqueness of $ \bar{\varphi} $ follows from the fact that $ \varphi $ is separately UWOT-continuous and $ \A $ is UWOT-dense in $ \bar{\A} $.
	
For $ 1\leq k \leq n $, suppose that $ \varphi_0, \varphi_1, ..., \varphi_{k-1} $ have been constructed. For $ j \neq k $ let $ a_j\in \A $ be fixed. The linear functional
$$ f_k:\A \rightarrow \mathbb{C};a \mapsto \varphi_{k-1}(a_1, ...a_{k-1}, a ,a_{k+1}, ..., a_n ) $$ is UWOT-continuous (and so $ \Arrowvert . \Arrowvert_{\tau} $-continuous on $ \A_1 $) and $$\Arrowvert f_k \Arrowvert \leq \max
\{\Arrowvert \varphi_{k-1} \Arrowvert, \Arrowvert a_1 \Arrowvert, ..., \Arrowvert a_{k-1} \Arrowvert, \Arrowvert a_{k+1} \Arrowvert, ..., \Arrowvert a_n \Arrowvert \}. $$
By Kaplansky density theorem, $ f_k $ extends without change of norm to an UWOT-continuous, smooth functional $ \bar{f_k} $ on $ \bar{\A} $.
	
Now we define $ \varphi_k(a_1, ..., a_k, ..., a_n )=\bar{f_k}(a_k) $. Clearly
$ \varphi_k $ is a bounded $ n $-linear  form on $ \underbrace{\bar{\A} \times \bar{\A}\times ... \times \bar{\A}}_{k-times} \times \A \times ... \times \A $
	that extends $ \varphi_{k-1} $ without change of norm and it is UWOT-continuous and smooth in its first $ k^{\text{th}} $ argument. We will show that $ \varphi_k  $ is UWOT-continuous
	and smooth in its other arguments for $ a_k \in \bar{\A}\setminus \A.$ Let $ 1\leq j \leq n $ with $ j \neq k $ and fix $ a_i $ for all $ i \neq j,k $ with $ a_i \in \bar{\A} $ for $ i<k $ or $ a_i \in \A $ for $ i>k $. Let $\B=\bar{\A}  $ if $ j<k $ and $\B=\A  $ if $ j>k $. Let  $ \psi:\A \times \B \rightarrow \mathbb{C} $ be the bounded bilinear form defined by $ \psi(a_k,a_j)=\varphi_{k-1} (a_1, ..., a_n)=\varphi_{k} (a_1, ..., a_n) $. By  assumption on $ \varphi_{k-1} $, $ \psi $ is a separately UWOT-continuous, smooth form so by Lemma \ref{L2} it extends uniquely to a bounded bilinear  smooth form $ \bar{\psi}:\bar{\A} \times \B \rightarrow \mathbb{C} $ which is separately UWOT-continuous. Since both $ \bar{\psi}(a_k,a_j) $ and $ \varphi_k(a_1, ..., a_n) $ are UWOT-continuous in the variable $ a_k \in \bar{\A} $ and they agree on $ \A $, it follows that
	$ \bar{\psi}(a_k,a_j) = \varphi_k(a_1, ..., a_n) $ on $ \bar{\A} \times \B $. This shows that for each $ a_k \in \bar{\A} $, the map $ \varphi_k $ is  UWOT-continuous
	 and smooth in $ a_j \in \B $, because $\bar{\psi}  $ has these properties.
	
	 Case 2. For each $ \xi \in  X_* $
	the bounded $ n $-linear   form $$ \rho_\xi:\A \times\A \times ... \times \A \rightarrow \mathbb{C};(a_1, ..., a_n)\mapsto \langle \varphi (a_1, ..., a_n), \xi \rangle$$
	is smooth and separately UWOT-continuous. Hence, by Case 1, it  extends  uniquely (without change of norm) to a separately UWOT-continuous, smooth $ n $-linear form $\bar{ \rho_\xi} $ on $ \bar{\A} \times\bar{\A} \times ... \times \bar{\A} $. Thus, for every $a_1, ..., a_n\in \bar{\A} $, the map $ \xi \mapsto \bar{ \rho_\xi}(a_1, ..., a_n) $ is a bounded linear functional on $ X_*$ and so belongs to $ X= (X_*)^* $. This defines a map $ \bar{\varphi}$ satisfying $ \Arrowvert \bar{\varphi} \Arrowvert = \Arrowvert \varphi \Arrowvert $. The smoothness and UWOT-continuity of $ \bar{\varphi}$ follows from the smoothness and UWOT-continuity of $\bar{ \rho_\xi} $.
\end{proof}

{\it Proof of Theorem \ref{t1}}. This is immediate by Lemma \ref{L3}, because the restriction map $\mathcal{H}^n_{sw}(\bar{\A},X) \rightarrow  \mathcal{H}^n_{sw}(\mathcal{A},X) $ is an isomorphism.\qed

\begin{rem}\label{r2}
		Let $ \A\subseteq B(\h) $ be a $ C^* $-algebra with a faithful normal trace on the UWOT-closure   $ \bar{\A} $ of $\A$ and let $ X $ be a normal dual $ \bar{\A} $-bimodule. If $ \pi $ is the universal representation of $\A$, then
		it is well known \cite{dix2} that  there is a projection  $ p $ in the center of the UWOT-closure $ \overline{\pi(\A)} $ of $\pi(\A) $ and an isomorphism $ \theta:p\overline{\pi(\A)}\rightarrow \bar{\A} $ such that
		\begin{eqnarray}\label{*}
		\theta(p\pi(a))=a \quad\text{and}\quad \theta(pb)=\pi^{-1}(b) \quad (a \in \A, b \in \overline{\pi(\A)} ).
		\end{eqnarray}
		By \cite[III. 2.2.12]{bl}, $ \theta $ is a homeomorphism in UWOT.
		Therefore,
		 $ X $ may be regarded as a normal dual $ \overline{\pi(\A)} $-bimodule with the following actions inherited from the actions of $ \bar{\A}  $ on $ X, $
		 \begin{eqnarray}\label{1}
		 b \cdot x:= \theta(pb)\cdot x \quad\text{and}\quad x \cdot b :=x \cdot \theta(pb) \quad (x \in X, b \in \overline{\pi(\A)} ).
		 \end{eqnarray}
		
		 In this case, every faithful normal trace $ \tau $ on $ \bar{\A}  $ induces a faithful normal trace $ \tau^\prime $ on $ \overline{\pi(\A)} $ defined by $\tau^\prime(\pi(a))=\tau\theta(p\pi(a)), \ a \in \A, $ such that for each  net $ (a_i)\subseteq \A_1$, $ \Arrowvert  a_i \Arrowvert_\tau\rightarrow 0 $ if and only if $ \Arrowvert  \pi (a_i) \Arrowvert_{\tau^{\prime}}\rightarrow 0 $.
\end{rem}
\begin{prop}\label{p}
	With the assumptions of Remark \ref{r2}, there are bounded linear maps
	$$ T_n:BL^n_{s}(\A,X)\rightarrow BL^n_{sw}(\overline{\pi(\A)},X),$$
	$$ S_n:BL^n_{sw}(\overline{\pi(\A)},X)\rightarrow BL^n_{sw}(\bar{\A},X),$$
    $$ W_n:BL^n_{sw}(\overline{\pi(\A)},X)\rightarrow BL^n_{s}(\A,X),$$	
	such that
	
	(i) $ \delta^n_{sw}T_n=T_{n+1}\delta^n_{s} $ and $ \delta^n_{sw}S_n=S_{n+1}\delta^n_{sw} $ such that the following internal diagrams are commutative:
	\begin{center}
		\begin{tikzpicture}
		\matrix [matrix of math nodes,row sep=1cm,column sep=1cm,minimum width=1cm]
		{
			|(A)| \displaystyle BL^n_{s}(\A,X)  &   |(B)|  BL^{n+1}_{s}(\A,X)    \\
			|(C)|  BL^n_{sw}(\overline{\pi(\A)},X)    &   |(D)|  BL^{n+1}_{sw}(\overline{\pi(\A)},X) \\
		    |(E)|    BL^n_{sw}(\bar{\A},X)   &  |(F)|   BL^{n+1}_{sw}(\bar{\A},X).  \\
		};
		\draw[->]    (A)-- node [above] { $ \delta^n_{s} $}(B);
		\draw[->]   (A)--  node [left] { $ T_n $} (C);
		\draw[->]  (C)-- node [above]   { $\delta^n_{sw} $}(D);
		\draw[->]  (B)-- node [right]  {$ T_{n+1} $}(D);
    	\draw[->]  (C)-- node [left]   { $S^n $}(E);
		\draw[->]  (E)-- node [above]   { $\delta^n_{sw} $}(F);
		\draw[->]  (D)-- node [right]   { $S_{n+1} $}(F);
		\end{tikzpicture}
	\end{center}
	
	(ii) If $ \B $ is a $ C^* $-subalgebra of $ \A $, then $ T_n $ maps $ \B $-modular maps to $\overline{\pi(\B)}$-modular maps and $ S_n $ and $W_n$ map $\overline{\pi(\B)}$-modular maps to maps.
	
	(iii) The map $ S_nT_n $ is a projection from $ BL^n_{s}(\A,X)  $ onto $ BL^n_{sw}(\A,X)  $.
	
	(iv) If $\C  $ is the $ C^* $-algebra generated by $ 1 $ and $ p $, the minimal projection in $ \overline{\pi(\B)} $ with $ \overline{\pi(\B)}\cdot p= \bar{\A}$ discussed in Remark \ref{r2}, and if $ \psi \in  BL^n_{sw}(\overline{\pi(\A)},X:\C) $, then $$ W_n\psi=S_n\psi \in  BL^n_{sw}(\bar{\A},X). $$
	
	(v) $ W_nT_n $ is the identity map on $  BL^n_{s}(\A,X). $
\end{prop}
\begin{proof}
For the projection $p$ as in Remark \ref{r2}, we have  $  p\cdot x=x\cdot p=x $,  for every $ x\in X.$  Also, for each $ b_1,..., b_n \in \pi(\A) $ and $ \varphi \in  BL^n_{s}(\A,X) $ the equality
	\begin{eqnarray}\label{2}
		\varphi_1(b_1,..., b_n)= \varphi(\theta(b_1), ..., \theta(b_n))
	\end{eqnarray}
	defines an element $ \varphi_1\in BL^n_{s}(\pi(\A),X)$. The map $ \varphi_1 $ is smooth because on bounded sets  UWOT agrees with  $  \Arrowvert . \Arrowvert_\tau $-topology and $ \theta $ is a UWOT-continuous homeomorphism. Since $ \pi $ is the universal representation of $ \A $, by \cite[Theorem 2.4]{tak}, each continuous linear functional on $ \pi(\A) $ is UWOT-continuous. Hence, $ \varphi_1 $ is separately UWOT-weak*-continuous, that is, $ \varphi_1\in BL^n_{sw}(\pi(\A),X)$.
	Therefore by Lemma \ref{L3}, $ \varphi_1 $ extends uniquely to some $ \tilde{\varphi}_1\in BL^n_{sw}(\overline{\pi(\A)},X)$ without change of norm. By Remark \ref{r2}, the map $ \tilde{\varphi}_1 $ is smooth. Now we define
	$$ T_n:BL^n_{s}(\A,X)\rightarrow BL^n_{sw}(\overline{\pi(\A)},X)$$
	by $ T_n\varphi=\tilde{\varphi}_1. $ It is easy to see that $ T_n $ is an isometry.
	If $\varphi\in BL^n_{s}(\A,X)  $ and $ b_1,..., b_{n+1} \in \pi(\A) $, then  the definition of $ T_n $ combined with the equations (\ref{1}) and (\ref{2}) yields
	\begin{align*}
		\delta^n_{sw}T_n\varphi(b_1,..., b_{n+1})
		&=\theta (pb_1)\cdot\varphi(\theta(pb_2),...,\theta(pb_{n+1}))      \\
		&+\sum_{j=1}^{n}(-1)^j\varphi(..., \theta(pb_j)\theta(pb_{j+1}),...)\\
		&+(-1)^{n+1}\varphi(\theta(pb_1),...,\theta(pb_n))\cdot\theta(pb_{n+1})\\
		&=\theta (pb_1)\cdot\varphi(\theta(pb_2),...,\theta(pb_{n+1}))      \\
		&+\sum_{j=1}^{n}(-1)^j\varphi(..., \theta(pb_jb_{j+1}),...)\\
		&+(-1)^{n+1}\varphi(\theta(pb_1),...,\theta(pb_n))\cdot\theta(pb_{n+1})\\
		&=T_{n+1}\delta^n_{s}\varphi(b_1,..., b_{n+1}).
	\end{align*}
	We use the fact that $ p $ is a central projection. Both the maps $ \delta^n_{sw}T_n\varphi $ and $ T_{n+1}\delta^n_{s}\varphi $ are separately UWOT-weak*-continuous, hence  $$ \delta^n_{sw}T_n\varphi(b_1,..., b_{n+1})=T_{n+1}\delta^n_{s}\varphi (b_1,..., b_{n+1}), $$  for every $ b_1,..., b_{n+1} \in \overline{\pi(\A)}. $ Thus $ \delta^n_{sw}T_n=T_{n+1}\delta^n_{s} $.
	
	If $ \B $ is a $ C^* $-subalgebra of $ \A $ and $ \varphi\in BL^n_{s}(\A,X:\B) $, then it follows from the equalities  $  p\cdot x=x\cdot p=x, $ for all $ x\in X,$ that
	 $ T_n\varphi \in BL^n_{sw}(\overline{\pi(\A)},X:\overline{\pi(\B)}) $: for instance, if
	$ a_1,..., a_n \in  \A  $ with $ b_j=\pi(a_j) $ and $ b\in \B $, then
	\begin{align*}
		T_n\varphi(b_1,...,b_j\pi(b),b_{j+1},..., b_{n})
		&=\varphi(\theta (pb_1),..., \theta(pb_j\pi(b)),...,\theta(pb_{n}))      \\
		&=\varphi(\theta (pb_1),..., a_jb,a_{j+1},...,\theta(pb_{n}))      \\
		&=\varphi(\theta (pb_1),..., a_j,ba_{j+1},...,\theta(pb_{n}))      \\
		&=\varphi(\theta (pb_1),..., \theta(pb_j),\theta(p\pi(b)b_{j+1}),...,\theta(pb_{n}))      \\
		&=T_n\varphi(b_1,...,b_j,\pi(b)b_{j+1},..., b_{n}).
	\end{align*}
	By the UWOT-weak*-continuity of the maps involved, the above calculation holds for each $ b_j\in \overline{\pi(\A)}.$ The calculation of the other cases is similar.
	
Next we define the map $ S_n $. For every $ \psi \in BL^n_{sw}(\overline{\pi(\A)},X), $ define $$ S_n (\psi)(a_1,..., a_{n})=\psi(\theta^{-1}(a_1), ..., \theta^{-1}(a_n)) \;\ (a_i\in \bar{\A}). $$
 	Since $ \psi $ and $ \theta^{-1} $ are UWOT-continuous, $ S_n\psi $ is normal and  Remark \ref{r2} implies that it is a smooth map. Hence, $ S_n $ maps
	$ BL^n_{sw}(\overline{\pi(\A)},X) $ into $ BL^n_{sw}(\bar{\A},X) $ and $ \Arrowvert S_n \Arrowvert \leq 1. $ By (\ref{*}), $ \theta(p\theta^{-1}(a))=a $,
	$ \theta^{-1}(a)\cdot x =\theta(p\theta^{-1}(a)) \cdot x =a \cdot x $ and
	$ x \cdot \theta^{-1}(a)=x \cdot a $ for all $ a \in \A $ and $ x \in X. $
	Hence,
	\begin{align*}
		S_{n+1}\delta^n_{sw}\psi(a_1,..., a_{n+1})
		&=\delta^n_{sw}\psi(\theta^{-1}(a_1),...,\theta^{-1}( a_{n+1}) )    \\
		&=a_1\cdot\psi(\theta^{-1}(a_2),...,\theta^{-1}( a_{n+1}) )    \\
		&+\sum_{j=1}^{n}(-1)^j\psi(\theta^{-1}(a_1),...,\theta^{-1}(a_j a_{j+1}),..., \theta^{-1}( a_{n+1}))\\
		&+(-1)^{n+1}\psi  (\theta^{-1}(a_1),...,\theta^{-1}( a_{n}) )\cdot a_{n+1}    \\
		&=\delta^n_{sw}S_n\psi(a_1,..., a_{n+1}),
	\end{align*}
	for every $ a_1,..., a_{n+1}\in \A $. By the normality of the maps involved, the equality holds on $ \bar{\A} $, that is, $ \delta^n_{sw}S_n=S_{n+1}\delta^n_{sw} $. Clearly $ S_n\psi $ is a $ \bar{\A} $-module map, whenever $ \psi $ is a $ \overline{\pi(\B)} $-module map.
	
	The map $ W_n:BL^n_{sw}(\overline{\pi(\A)},X)\rightarrow BL^n_{s}(\A,X), $ defined by $ W_n\psi(a_1,..., a_{n})= \psi (\pi (a_1),..., \pi ( a_n)) $ is a continuous linear map with $ \Arrowvert  W_n \Arrowvert \leq 1 $. Note that by Remark \ref{r2}, the smoothness of $ \psi \in BL^n_{sw}(\overline{\pi(\A)},X)$ implies the smoothness of $ W_n\psi $.
	
	 If $ \varphi \in BL^n_{s}(\A,X) $, then by  (\ref{*}) and  (\ref{2}),
	\begin{align*}
		W_nT_n\varphi(a_1,..., a_{n})
		&=T_n\varphi(\pi (a_1),..., \pi (a_{n}))      \\
		&=\varphi(\theta (p\pi (a_1)),...,\theta(p\pi (a_{n})))      \\
		&=\varphi(a_1,..., a_{n}),
	\end{align*}
	which proves $ (v) $.

To prove $ (iv) $, let $ \psi\in  BL^n_{sw}(\overline{\pi(\A)},X:\C) $. Since $ p^2=p $ in the center of $ \overline{\pi(\A)} $	and $ \psi $ is a $ \C $-module map, we have
	\begin{align*}
		W_n\psi(a_1,..., a_{n})
		&=\psi(\pi (a_1),..., \pi (a_{n}))      \\
		&=\psi(\pi (a_1),..., \pi (a_{n})) \cdot p  \quad \quad\quad( \hbox{since}\; p\cdot x=x\cdot p=x )   \\
		&=\psi(\pi (a_1)p,..., \pi (a_{n})p) \\
		&=\psi(\theta^{-1} (a_1),..., \theta^{-1} (a_{n})) \quad \quad( \hbox{by}\: (\ref{*}), \;\: \theta^{-1} (a_i)=\pi (a_i)p)\\
		&=S_n\psi(a_1,..., a_{n}),
	\end{align*}
	as required. This finishes the proof.
	\end{proof}
	
\begin{lem}\label{5}
	Let $ \A\subseteq B(\h) $ be a $ C^* $-algebra with a faithful normal trace on   $ \bar{\A} $ and let $ X $ be a normal dual $ \bar{\A} $-bimodule. Then there is a bounded linear map $ J_n:BL^n_{s}(\A,X)\rightarrow BL^{n-1}_{s}(\A,X) $ with the following properties;
	
	(i) $ \Arrowvert J_n \Arrowvert \leq ((n+2)^n-1)/(n+1) $,
	
	(ii) if $ \varphi \in BL^n_{s}(\A,X) $ with $\delta^n_{s} \varphi=0, $ then
	$ \varphi-\delta^{n-1}_{s}J_n\varphi \in BL^n_{sw}(\A,X) $,
	
	(iii) if $ \B $ is a $ C^* $-subalgebra of $ \A $, then $ J_n $ maps $ BL^n_{s}(\A,X:\B ) $ into \\
	$ BL^{n-1}_{s}(\A,X:\B ) $.
	\end{lem}
\begin{proof}
	Let $ \pi $ be the universal representation of $ \A $ and $ p $ be the central projection in $ \overline{\pi(\A)} $ as in Remark \ref{r2}. The unitary subgroup consisting of the two elements $ \{1, 2p-1\} $ generates a two dimensional
	$ C^* $-subalgebra  $ \C $ in the center of $ \overline{\pi(\A)} $. By the averaging techniques similar to \cite[Lemma 3.2.4(a)]{SS}, there is a continuous linear map $ K_n:BL^n_{sw}(\overline{\pi(\A)},X)\rightarrow BL^{n-1}_{sw}(\overline{\pi(\A)},X)$ such that $ (I-\delta^{n-1}_{sw}K_n)\psi $ is a
	$ \C $-module map, for each $ \psi \in BL^n_{sw}(\overline{\pi(\A)},X) $ with
	$ \delta^{n}_{sw}\psi=0 $.
	
	Let $ T_n, W_n $ be as in Proposition \ref{p}. Define
	$$ J_n:BL^n_{s}(\A,X)\rightarrow BL^{n-1}_{s}(\A,X) $$ by $ J_n=W_{n-1}K_nT_n, $
	then the following diagram is commutative:
	\begin{center}
		\begin{tikzpicture}
		\matrix [matrix of math nodes,row sep=1cm,column sep=1cm,minimum width=1cm]
		{
			|(A)| \displaystyle BL^n_{s}(\A,X)  &   |(B)| BL^{n}_{sw}(\overline{\pi(\A)},X)    \\
			|(C)|  BL^{n-1}_{s}(\A,X)    &   |(D)|  BL^{n-1}_{sw}(\overline{\pi(\A)},X). \\
				};
		\draw[->]    (A)-- node [above] { $ T_n $}(B);
		\draw[->]   (A)--  node [left] { $ J_n $} (C);
		\draw[<-]  (C)-- node [above]   { $W_{n-1} $}(D);
		\draw[->]  (B)-- node [right]  {$ K_n $}(D);
		
				\end{tikzpicture}
	\end{center}
	By \cite[Lemma 3.2.4]{SS}, $ \Arrowvert J_n \Arrowvert \leq ((n+2)^n-1)/(n+1),  $ and by Proposition \ref{p}, $ J_n $ takes $\B  $-module maps to $\B  $-module maps, and this proves (i) and (iii). Since $ W_nT_n $ is the identity on $ BL^n_{s}(\A,X) $, the equation $ \delta^{n-1}_{s}W_{n-1}= W_n\delta^{n-1}_{sw} $ implies that
	$$ \varphi-\delta^{n-1}_{s}J_n\varphi=\varphi-\delta^{n-1}_{s}W_{n-1}K_nT_n\varphi=W_n(T_n\varphi-\delta^{n-1}_{sw}K_nT_n\varphi). $$
	Now Proposition \ref{p}(i) implies that $ \delta^{n}_{sw}T_n\varphi = T_{n+1}\delta^{n}_{s}\varphi=0 $. Hence, $ T_n\varphi-\delta^{n-1}_{sw}K_nT_n\varphi $ is a $ \C $-module map.
Proposition \ref{p}(iv)	asserts that $ W_n $ takes $ \C $-module smooth maps
 to smooth normal maps, so $ \varphi-\delta^{n-1}_{s}J_n\varphi $ is a smooth normal map. This completes the proof.
\end{proof}

{\it Proof of Theorem \ref{t2}}. 	Consider the natural embedding
	$$ Q_n: BL^{n}_{sw}(\A,X)\rightarrow BL^n_{s}(\A,X). $$
	If $ \varphi\in BL^{n}_{sw}(\A,X) $ with $ \varphi=\delta^{n-1}_{s}\psi $ for some $ \psi\in BL^{n-1}_{s}(\A,X) $, then by Proposition \ref{p}(i) and (iii), $ \varphi=S_nT_n \delta^{n-1}_{s}\psi=\delta^{n-1}_{sw}S_{n-1}T_{n-1}\psi $. Therefore, $ Q_n $ induces an injective map $\tilde{Q}_n: \mathcal{H}^n_{sw}(\mathcal{A},X)\rightarrow\mathcal{H}^n_{s}(\A,X), $ which is surjective by Lemma \ref{5}. Hence, $ \mathcal{H}^n_{sw}(\mathcal{A},X)=\mathcal{H}^n_{s}(\A,X). $\qed

\vspace{.3cm}
 Theorems \ref{t1} and \ref{t2} yield the following result.

\begin{cor}\label{c1}
		Let $ \A\subseteq B(\h) $ be a $ C^* $-algebra with a faithful normal trace on   $ \bar{\A} $ and let $ X $ be a normal dual $ \bar{\A} $-bimodule. Then, for every $ n \in  \mathbb{N}$ we have
		$$ \mathcal{H}^n_{s}(\mathcal{A},X)=\mathcal{H}^n_{sw}(\bar{\A},X). $$
	\end{cor}
	A. Galatan and S. Popa in \cite{SS} showed that for a von Neumann algebra $ \mathcal{M} $ with a faithful normal trace $ \tau $ and a normal dual operatorial $ \mathcal{M}$-bimodule  $ X $ we have
	\begin{eqnarray}\label{popa}
		\mathcal{H}^1_{s}(\mathcal{M},X)=\mathcal{H}^1(\mathcal{M} ,X_{\tau}).
	\end{eqnarray}

	A Banach $ \mathcal{M}$-bimodule $ X $ is called operatorial if for  every projection $ p \in \mathcal{M}$ and $ x \in X $,
	$$ \Arrowvert p\cdot x\cdot p + (1-p)\cdot x \cdot (1-p) \Arrowvert = \max \{  \Arrowvert  p\cdot x\cdot p  \Arrowvert , \Arrowvert (1-p)\cdot x \cdot (1-p) \Arrowvert\}.$$
	
	By \cite[Proposition  2.2]{gp}, every smooth derivation of a von Neumann algebra $\mathcal{M}  $ to a  dual $ \mathcal{M} $-bimodule is normal, that is, UWOT-weak*-continuous. Therefore, $ \mathcal{H}^1_{s}(\mathcal{M},X)=\mathcal{H}^1_{sw}(\mathcal{M},X) $. Hence, combining \cite[Theorem  3.5]{gp} with  (\ref{popa})
	yields
	\begin{eqnarray}\label{popa2}
		\mathcal{H}^1_{sw}(\mathcal{M},X)=\mathcal{H}^1_{w}(\mathcal{M} ,X_{\tau}).
	\end{eqnarray}
	
	We  use this fact to prove the main  result of this paper.

\vspace{.3cm}
{\it Proof of Theorem \ref{main}}.
	\begin{align*}
		\mathcal{H}^1_{s}(\mathcal{A},X)
		&=\mathcal{H}^1_{sw}(\bar{\A},X)  \quad \quad\quad( \hbox{by Corollary \ref{c1} })    \\
		&=\mathcal{H}^1_{w}(\bar{\A} ,X_{\tau})  \quad \quad\quad( \hbox{by  (\ref{popa2})})   \\
		&=\mathcal{H}^1(\A ,X_{\tau}) \,\quad \quad\quad( \hbox{by \cite[Theorem  3.3.1]{SS} }).\qed
	\end{align*}

We don't know if the Banach $ \mathcal{A}$-bimodule $B(\mathcal{A}, X)$ of bounded $\mathcal{A}$-bimodule maps from $\mathcal{A}$ to an operatorial Banach $ \mathcal{A}$-bimodule $ X $ is again  operatorial. If this is the case, by a standard reduction of order argument for cohomologies, one could conclude that $\mathcal{H}^n_{s}(\mathcal{A},X)
		=\mathcal{H}^n(\A ,X_{\tau})$, for each $n\geq 1$.


\begin{thebibliography}{}

\bibitem{ak} C. A. Akemann,  The dual space of an operator algebra, \textit{Trans. A.M.S.}, \textbf{126} (2) (1967), 286-302.
\bibitem{ren} L. Bing-Ren, \textit{Introduction to Operator Algebras}, World Scientific, Singapore, 1992.
\bibitem{bl} B. Blackadar, \textit{Operator Algebras, Theory of C*-Algebras and von Neumann Algebras}, Encyclopedia of Mathematical Sciences, Vol. 122, Springer-Verlag, Berlin, 2006.
\bibitem{ces} E. Christensen, E.G. Effros and A.M. Sinclair, Completely bounded
multilinear maps and C*-algebraic cohomology, {\it Invent. Math.} {\bf 90}
(1987), 279-296.
\bibitem{c1} A. Connes, Outer conjugacy classes of automorphisms of factors
{\it Ann. Sci. Ecole Norm. Sup.} {\bf 8} (1975), 383-419.
\bibitem{c2} A. Connes, Classification of injective factors
{\it Ann. Math.} {\bf 104} (1976), 73-115.
\bibitem{c3} A. Connes, Periodic automorphisms of the hyperfinite factor of type $II_1$, {\it Acta Sci. Math.} {\bf 39} (1977),  39-66.
\bibitem{dix2} J. Dixmier, $ Les \; alg\grave{e}bras\; d'op\acute{e}rateurs \; dans\; l'espace\; Hilbertien,$ Gauthier-Villars, Paris, 1969.
\bibitem{el} G. A. Elliott,  On the classification of inductive limits of sequences of semi-simple finite dimensional algebras , \textit{J. Algebra},  \textbf{38} (1976), 29-44.
\bibitem{gp} A. Galatan and S. Popa,  Smooth bimodules and cohomology of $ \mathrm{II}_1 $  factors, \textit{J. Inst. Math. Jussieu},  \textbf{16}(1) (2017), 155-187.
\bibitem{h} U. Haagerup, A new proof of the equivalence of injectivity and hyperfiniteness for factors on a separable Hilbert space, {\it Jour. Funct. Anal.} {\bf 62} (1985), 160-201.
\bibitem{h1} G. Hochschild, On the cohomology groups of an associative
algebra, {\it Ann. Math.} {\bf 46} (1945), 58-67.
\bibitem{h2} G. Hochschild, On the cohomology theory for associative algebras,
{\it Ann. Math.} {\bf 47} (1946), 568-579.
\bibitem{h3} G. Hochschild, Cohomology and representations of associative
algebras, {\it Duke Math. J.} {\bf 14} (1947), 921-948.
\bibitem{jkr} B.E. Johnson, R.V. Kadison and J.R. Ringrose, Cohomology of
operator algebras III. Reduction to normal cohomology, {\it Bull. Soc.
Math. France} {\bf 100} (1972), 73-96.
\bibitem{k} R.V. Kadison, Derivations of operator algebras, {\it Ann. Math.} {\bf 83} (1966), 280-293.
\bibitem{ps} F. Pop and R.R. Smith, Cohomology for certain finite factors,
{\it Bull. London Math. Soc} {\bf 26} (1994), 303–308.
\bibitem{pv} S. Popa and S. Vaes, Unique Cartan decomposition for II1 factors arising from arbitrary actions of free groups, {\it Acta Math.}
{\bf  212} (2014), 141-198.
\bibitem{s} S. Sakai, Derivations of $W^*$-algebras, {\it Ann. Math.} {\bf 83} (1966), 273-279.
\bibitem{SS} A. M. Sinclair and R. R. Smith, \textit{Hochschild Cohomology of von Neumann Algebras},
London Mathematical Society Lecture Note Series, Vol. 203, Cambridge University
Press, Cambridge, 1995.
\bibitem{tak} M. Takesaki, \textit{Theory of Operator Algebras I}, Encyclopedia of Mathematical Sciences,
Vol. 124, Springer-Verlag, Berlin, 2002.
\end{thebibliography}
\end{document}